\documentclass[12pt]{article}
\usepackage{amsmath,amssymb}
\begin{document}
\newtheorem{theorem}{Theorem}
\newtheorem{lemma}[theorem]{Lemma}
\newtheorem{corollary}[theorem]{Corollary}
\newtheorem{definition}[theorem]{Definition}
\newtheorem{example}[theorem]{Example}
\pagenumbering{roman}
\renewcommand{\thetheorem}{\thesection.\arabic{theorem}}
\renewcommand{\thelemma}{\thesection.\arabic{lemma}}
\newenvironment{proof}{\noindent{\bf{Proof.\/}}}{\hfill$\blacksquare$\vskip0.1in}
\renewcommand{\thetable}{\thesection.\arabic{table}}
\renewcommand{\thedefinition}{\thesection.\arabic{definition}}
\renewcommand{\theexample}{\thesection.\arabic{example}}
\renewcommand{\theequation}{\thesection.\arabic{equation}}
\newcommand{\mysection}[1]{\section{#1}\setcounter{equation}{0}
\setcounter{theorem}{0} \setcounter{lemma}{0}
\setcounter{definition}{0}}
\newcommand{\mrm}{\mathrm}
\newcommand{\beq}{\begin{equation}}
\newcommand{\eeq}{\end{equation}}

\newcommand{\ben}{\begin{enumerate}}
\newcommand{\een}{\end{enumerate}}

\newcommand{\hc}{\hat{c}}
\newcommand{\hd}{\hat{d}}
\newcommand{\hp}{\hat{p}}
\newcommand{\hq}{\hat{q}}
\newcommand{\hu}{\hat{u}}
\newcommand{\hv}{\hat{v}}
\newcommand{\hP}{\widehat{P}}
\newcommand{\hQ}{\widehat{Q}}
\newcommand{\hF}{\widehat{F}}
\newcommand{\hG}{\widehat{G}}
\newcommand{\hH}{\widehat{H}}
\newcommand{\tT}{\widetilde{T}}
\newcommand{\btB}{\mathbf{\widetilde{B}}^{(s)}}
\newcommand{\bA}{\mathbf{A}}
\newcommand{\bB}{\mathbf{B}}
\newcommand{\cT}{\check{T}}
\newcommand{\C}{\mbox{$\mathbb{C}$}}
\title
{\bf A de Montessus Type Convergence Study for a Vector-Valued Rational Interpolation Procedure of  Epsilon Class}

\author
{Avram Sidi\\
Computer Science Department\\
Technion - Israel Institute of Technology\\ Haifa 32000, Israel\\
E-mail:  asidi@cs.technion.ac.il\\
http://www.cs.technion.ac.il/\~{}asidi}
\date{January 2017}
\bigskip\bigskip
\maketitle \thispagestyle{empty}
\newpage
\begin{abstract}
In a series of recent publications  of the author, three  interpolation procedures, denoted IMPE, IMMPE, and ITEA,
 were proposed for vector-valued functions $F(z)$, where $F : \C \to\C^N$,
and  their algebraic properties were studied. The convergence studies of two of the methods, namely, IMPE and IMMPE, were also carried out  as these  methods are being applied to meromorphic functions with simple poles, and de Montessus and K\"{o}nig type theorems for them were proved.
In the present work,
we concentrate on  ITEA. We study
its convergence properties as it is applied to meromorphic functions with simple poles,
and  prove de Montessus and K\"{o}nig type theorems analogous to those obtained for IMPE and IMMPE.
\end{abstract}

\vspace{1cm} \noindent {\bf Mathematics Subject Classification
2000:} 30E10, 30E15, 41A20, 41A25.

\vspace{1cm} \noindent {\bf Keywords and expressions:}
Vector-valued rational interpolation; Hermite interpolation; Newton interpolation formula; de Montessus theorem; K\"{o}nig theorem

 \thispagestyle{empty}

\newpage
\pagenumbering{arabic}
\section{Introduction and background} \label{se1}
  \setcounter{equation}{0} \setcounter{theorem}{0}
   In \cite{Sidi:2004:NAV}, the author developed three rational interpolation methods for vector-valued functions of a complex variable.   These methods were denoted IMPE, IMMPE, and ITEA. Some of the algebraic properties of these methods were already presented in \cite{Sidi:2004:NAV} while others were  explored in \cite{Sidi:2006:APS}, where it was also shown that the methods are symmetric functions of the points of interpolation and that they reproduce
   vector-valued
   rational functions exactly. In \cite{Sidi:2008:MTC}, \cite{Sidi:2008:IMPE-1}, and \cite{Sidi:2010:IMPE-2}, de Montessus and K\"{o}nig type convergence theories for IMMPE and IMPE, as these methods are applied to vector-valued meromorphic functions with simple poles, were presented. In this work, we treat the convergence properties of ITEA, as it is being applied to the same class of functions, and we prove de Montessus and K\"{o}nig type theorems analogous to those  for IMPE and IMMPE. As will become clear, following some necessary adjustments, the techniques of \cite{Sidi:2008:MTC} that were developed for analyzing IMMPE, will be  directly applicable when analyzing ITEA.

\section{Review of the algebraic properties of ITEA}
\label{se2}
To set the stage for later developments, and to fix the notation as well, we start
with a brief description of the developments in \cite{Sidi:2004:NAV}  and \cite{Sidi:2006:APS}.

Let $z$ be a complex variable and let $F(z)$ be a vector-valued
function such that $F:\C\to \C^N$. Assume that $F(z)$ is defined
on a bounded open set $\Omega\subset \C$ and consider the problem of
interpolating $F(z)$ at some of the points $\xi_1,\xi_2,\ldots,$
in this set. We do not assume that the $\xi_i$ are necessarily
distinct. The general picture is described in the next paragraph:

Let $a_1,a_2,\ldots,$ be distinct complex numbers, and order the $\xi_i$ such that
\begin{align} \label{eq:2-12}
&\xi_1=\xi_2=\cdots=\xi_{r_1}=a_1 \nonumber \\
&\xi_{r_1+1}=\xi_{r_1+2}=\cdots=\xi_{r_1+r_2}=a_2 \nonumber \\
&\xi_{r_1+r_2+1}=\xi_{r_1+r_2+2}=\cdots=\xi_{r_1+r_2+r_3}=a_3 \nonumber \\
&\mbox{and so on.}
\end{align}

Let $G_{m,n}(z)$ be the vector-valued polynomial (of degree at
most $n-m$) that interpolates $F(z)$ at the points
$\xi_m,\xi_{m+1},\ldots,\xi_n$ in the generalized Hermite sense.
Thus, in Newtonian form, this polynomial is given as in (see,
e.g., Stoer and Bulirsch \cite[Chapter 2]{Stoer:2002:INA} or
Atkinson \cite[Chapter 3]{Atkinson:1989:INA})

\beq \label{eq:2-2}G_{m,n}(z)=\sum^n_{i=m}F[\xi_m,\xi_{m+1},\ldots,\xi_i]\prod^{i-1}_{j=m}(z-\xi_j).\eeq
Here, $F[\xi_r,\xi_{r+1},\ldots,\xi_{r+s}]$ is the divided
difference of order $s$ of $F(z)$ over the set of points
$\{\xi_r,\xi_{r+1},\ldots,\xi_{r+s}\}$.
 Obviously,
$F[\xi_r,\xi_{r+1},\ldots,\xi_{r+s}]$ are all vectors in $\C^N$.

Let us  define the scalar polynomials $\psi_{m,n}(z)$ via
\begin{equation}\label{eq:2-1}
\psi_{m,n}(z)=\prod^{n}_{r=m}(z-\xi_{r}),\quad n\geq m\geq 1;
\quad \psi_{m,m-1}(z)=1, \quad m\geq 1.
\end{equation}
Let us also define the vectors $D_{m,n}$ via
\begin{equation}\label{eq:2-1a}
D_{m,n}=F[\xi_m,\xi_{m+1},\ldots,\xi_n],\quad n\geq m.
\end{equation}
With this notation, we can rewrite (\ref{eq:2-2}) in the form
\begin{equation}\label{eq:2-2a}
G_{m,n}(z)=\sum^n_{i=m}D_{m,i}\,\psi_{m,i-1}(z).
\end{equation}

Then the  vector-valued rational function
   $R_{p,k}(z)$ from ITEA that interpolates $F(z)$ at $\xi_1,\ldots,\xi_p$ in the sense of Hermite is defined as in
  \beq\label{eq:2-5}
R_{p,k}(z)=\frac{U_{p,k}(z)}{V_{p,k}(z)}=
\frac{\sum^k_{j=0}c_j\,\psi_{1,j}(z)\,G_{j+1,p}(z)}
{\sum^k_{j=0}c_j\,\psi_{1,j}(z)},\eeq
the scalars $c_0,c_1,\ldots,c_k$ being  determined  by the requirement
\begin{equation}\label{eq:3-9}
\bigg(q,\sum^{k}_{j=0}c_jD_{j+1,p+i}\bigg)=0, \quad
i=1,\ldots,k;\quad c_k=1,
\end{equation}
where $(\cdot\,,\cdot)$ is an inner product
and $q$ is some fixed nonzero vector
in $\C^N$.
 Clearly, \eqref{eq:3-9} results in  the linear system
\begin{equation}\label{IMMPE}  \sum^k_{j=0}u_{i,j}c_j=-u_{i,k},\quad i=1,\ldots,k;\quad c_k=1;
\quad u_{i,j}=\big(q,D_{j+1,p+i}\big),
\end{equation}
a unique solution for which exists provided
 \begin{equation}\label{unique}
\begin{vmatrix} u_{1,0}&u_{1,1}&\cdots&u_{1,k-1}\\
u_{2,0}&u_{2,1}&\cdots&u_{2,k-1}\\
\vdots&\vdots& &\vdots\\
u_{k,0}&u_{k,1}&\cdots&u_{k,k-1}\end{vmatrix}\neq 0.
\end{equation}

Combining \eqref{eq:2-5} and \eqref{IMMPE}, we obtain the following determinant representation for $R_{p,k}(z)$ from ITEA, with $u_{i,j}=(q,D_{j+1,p+i})$, $i\geq1$, $j\geq0$:
\beq \label{eq:16-1}
R_{p,k}(z)=\frac{P(z)}{Q(z)}=\frac
{\begin{vmatrix}\displaystyle \psi_{1,0}(z)\,G_{1,p}(z)&
\displaystyle \psi_{1,1}(z) \,G_{2,p}(z) &
\cdots& \displaystyle \psi_{1,k}(z)\, G_{k+1,p}(z)\\
u_{1,0}&u_{1,1}&\cdots&u_{1,k}\\
u_{2,0}&u_{2,1}&\cdots&u_{2,k}\\
\vdots&\vdots& &\vdots\\
u_{k,0}&u_{k,1}&\cdots&u_{k,k}\end{vmatrix}}
{\begin{vmatrix}
\displaystyle \psi_{1,0}(z)&
\displaystyle \psi_{1,1}(z)&
\cdots& \displaystyle \psi_{1,k}(z)\\
u_{1,0}&u_{1,1}&\cdots&u_{1,k}\\
u_{2,0}&u_{2,1}&\cdots&u_{2,k}\\
\vdots&\vdots& &\vdots\\
u_{k,0}&u_{k,1}&\cdots&u_{k,k}\end{vmatrix}}.
\end{equation}
Here, the numerator determinant $P(z)$ is vector-valued and is
defined by its expansion with respect to its first row. That is,
if $M_j$ is the cofactor of the term $\psi_{1,j}(z)$ in the
denominator determinant $Q(z)$, then
\begin{equation}\label{eq:4-3a}
R_{p,k}(z)=\frac{\sum^{k}_{j=0}M_j\,\psi_{1,j}(z)\,G_{j+1,p}(z)}
{\sum^{k}_{j=0}M_j\,\psi_{1,j}(z)}.
\end{equation}
[Note that this determinant representation  offers a
very effective tool for the algebraic and analytical study of $R_{p,k}(z)$.
As we will see
later in this work, it forms the basis of our convergence study.]

From \eqref{eq:2-5} and \eqref{eq:3-9}, it is clear that  the number of function evaluations [namely, (i)\,$F(\xi_i)$ in case the $\xi_i$ are distinct and (ii)\,$F(\xi_i)$ and some of its derivatives otherwise] that are needed to determine $R_{p,k}(z)$ is $p+k$, and these are based on $\xi_1,\ldots,\xi_{p+k}$. [This should be contrasted with the interpolants $R_{p,k}(z)$ that result from IMPE and IMMPE, which need $p+1$ function evaluations based on $\xi_1,\ldots,\xi_{p+1}$.]
\medskip

\subsubsection*{Remarks:}
\begin{enumerate}
\item
$R_{p,k}(z)=U_{p,k}(z)/V_{p,k}(z)$ from ITEA interpolates $F(z)$ at $\xi_1,\ldots,\xi_p$ in the sense of Hermite, provided $V_{p,k}(\xi_i)\neq0$ for all $i=1,\ldots,p.$
\item
 Note that $R_{p,k}(z)$, even with {\em arbitrary} $c_j$ in \eqref{eq:2-5},
interpolates $F(z)$ at $\xi_1,\ldots,\xi_p$ in the sense of Hermite,
provided $V_{p,k}(\xi_i)\neq0$ for all $i=1,\ldots,p.$
 However, the quality of $R_{p,k}(z)$ as an approximation to $F(z)$ in the $z$-plane depends heavily on how the $c_j$ are chosen. Thus, the methods IMPE, IMMPE, and ITEA choose the $c_j$ in special ways; as we have shown in \cite{Sidi:2008:MTC}, \cite{Sidi:2008:IMPE-1}, and \cite{Sidi:2010:IMPE-2},
the methods IMPE and IMMPE do provide very good approximations for meromorphic functions $F(z)$.
Here we prove that ITEA does too.
\end{enumerate}

We end this section by stating four  algebraic properties of ITEA. Of these, the first three were explored in \cite{Sidi:2006:APS}, while the forth is new:

\begin{enumerate}
\item {\em Limiting property:}
When $\xi_i$ all tend to $0$ simultaneously, it follows from the equations in \eqref{IMMPE} that
$R_{p,k}(z)$ tends to the  approximant $s_{n+k,k}(z)$  from the method STEA of Sidi \cite{Sidi:1994:RAP} as the latter is being applied to the Maclaurin  series of $F(z)$.\footnote{STEA approximants are obtained by applying  the topological epsilon  algorithm (TEA) of Brezinski \cite{Brezinski:1975:GTS} to the sequence of partial sums of the Maclaurin  series of $F(z)$.} Here $n=p-k$.

 \item {\em Symmetry property:}
 The denominator polynomial $V_{p,k}(z)=\sum^{k}_{j=0}c_j\psi_{1,j}(z)$ is a symmetric function of $\xi_1,\ldots,\xi_{p+k}$, which  go into its construction. $R_{p,k}(z)$ itself is a symmetric function of $\xi_1,\ldots,\xi_p$.\footnote{A function $f(x_1,\ldots,x_m)$ is  symmetric in
 $x_1,\ldots,x_m$ if $f(x_{i_1},\ldots,x_{i_m})=f(x_1,\ldots,x_m)$
 for every permutation $(x_{i_1},\ldots,x_{i_m})$ of
 $(x_1,\ldots,x_m)$.}

 \item  {\em Reproducing property:}
 If $F(z)=\widetilde{U}(z)/\widetilde{V}(z)$ is a vector-valued rational function with degree of numerator $\widetilde{U}(z)$ at most $p-1$ and degree of denominator $\widetilde{V}(z)$ equal to $k$ and if $F(\xi_i)$, $i=1,\ldots,p,$  are all defined, then $R_{p,k}(z)\equiv F(z)$.

\item {\em Projection property:}
In addition to  interpolating $F(z)$ at $\xi_1,\ldots,\xi_p,$  $R_{p,k}(z)$ also has  the following projection property:
$$ \big(q,F(z)-R_{p,k}(z)\big)\big|_{z=\xi_{p+i}}=0,\quad i=1,\ldots,k.$$
\end{enumerate}

Because ITEA and IMMPE, in producing the relevant $R_{p,k}(z)$, differ substantially (i)\,in the number of the $\xi_i$ they use and (ii)\,in the structure of the relevant scalars $u_{i,j}$,   it seems that their analyses should be different from each other.
Fortunately,  in this work, we  are able to overcome these obstacles and apply to ITEA the techniques  used for analyzing IMMPE, following some clever adjustments.

To keep things simple, in the sequel, we adopt the notation of \cite{Sidi:2008:MTC}, where we treated  IMMPE.
In order not to repeat the arguments of \cite{Sidi:2008:MTC} unnecessarily, we will keep our
treatment of ITEA short and will refer the reader to \cite{Sidi:2008:MTC} for
technical details.

\section{Technical preliminaries and error formula  when $F(z)$ is a vector-valued rational function}\label{se3}

\setcounter{equation}{0}
\setcounter{theorem}{0}
We start our study of ITEA for the case in which the function $F(z)$ is
a vector-valued rational function with simple poles, namely,
\begin{equation}\label{ratF}
F(z)=\sum^\mu_{s=1}\frac{v_s}{z-z_s}+u(z),
\end{equation}
where $u(z)$ is an arbitrary vector-valued polynomial,
$z_1,\ldots,z_\mu$ are distinct points in the complex plane,
and $v_1,\ldots,v_\mu$ are  some nonzero  vectors
in $\C^N$. \\

\subsection{Technical preliminaries}
The following  technical tools that were used  in \cite{Sidi:2008:MTC} will be used throughout this work too.

\begin{lemma}[\cite{Sidi:2008:MTC}, Lemma 3.2] \label{le:a2}
Let $Q_i(x)=\sum^i_{j=0} a_{ij} x^j$, with $a_{ii} \neq
0,\ i=0,1\ldots,n$, and let $x_i,\ i=0,1,\ldots,n$,
be arbitrary complex numbers. Then
\begin{equation} \label{eq:a33}
\left |
\begin{array}{cccc}
Q_0(x_0) & Q_0(x_1) & \cdots & Q_0(x_n)\\
Q_1(x_0) & Q_1(x_1) & \cdots & Q_1(x_n)\\
\vdots & \vdots & & \vdots\\
Q_n(x_0) & Q_n(x_1) & \cdots & Q_n(x_n)
\end{array}
\right |
= \left (\prod^n_{i=0}a_{ii} \right )V(x_0, x_1,\ldots,
x_n),
\end{equation}
where
$$V(x_0,x_1,\ldots,x_n)=\begin{vmatrix}1&1&\cdots&1\\ x_0&x_1&\cdots&x_n\\
\vdots&\vdots&&\vdots\\ x_0^n&x_1^n&\cdots&x_n^n\end{vmatrix}=\prod_{0\leq i < j \leq n}
(x_j-x_i)$$ is a Vandermonde determinant.
\end{lemma}

\begin{lemma}[\cite{Sidi:2008:MTC}, Lemma 3.3] \label{le1}
Let $\omega_a(z)=(z-a)^{-1}$. Then, $ \omega_a[\xi_m,\ldots,\xi_n]$, the divided
difference of $\omega_a(z)$ over the set of points $\{\xi_m,\ldots,\xi_n\}$, is given by
\begin{equation}\label{omega}
 \omega_a[\xi_m,\ldots,\xi_n]=-\frac{1}{\psi_{m,n}(a)}.
 \end{equation}
This is true whether the $\xi_i$ are distinct or not.
\end{lemma}

\begin{lemma}[\cite{Sidi:2008:MTC}, Lemma 3.4] \label{le2}
Let $F(z)$ be given as in \eqref{ratF}. Let $n-m>\deg(u)$. Then,
the following are true whether the $\xi_i$ are distinct or not:
\begin{itemize}
\item [\textnormal{(i)}]
 $D_{m,n}=F[\xi_m,\ldots,\xi_n]$ is given as in
\begin{equation}\label{dmn}
 D_{m,n}=-\sum^\mu_{s=1}\frac{v_s}{\psi_{m,n}(z_s)}.
\end{equation}
Therefore, we also have
\begin{equation}\label{qdmn}
\big(q,D_{m,n}\big)=-\sum^\mu_{s=1}\frac{(q,v_s)}{\psi_{m,n}(z_s)}.
 \end{equation}
 \item[\textnormal{(ii)}]
$F(z)-G_{m,n}(z)=\psi_{m,n}(z)F[z,\xi_m,\ldots,\xi_n]$ is given as in
\begin{equation}\label{F-G}
F(z)-G_{m,n}(z)=\psi_{m,n}(z)\sum^\mu_{s=1}\frac{v_s}{z-z_s}\,\frac{1}
{\psi_{m,n}(z_s)}.
\end{equation}
 \end{itemize}
\end{lemma}

\subsection{Error formula}
Using \eqref{eq:16-1}, \eqref{eq:4-3a}, and \eqref{F-G}, we can derive a determinant representation for the error $F(z)-R_{p,k}(z)$ as in the next lemma:
\begin{lemma}[\cite{Sidi:2008:MTC}, Lemma 3.5] \label{le:error}
Let
\begin{equation}\label{delta_j}
\Delta_j(z)=\psi_{1,j}(z)\big[F(z)-G_{j+1,p}(z)\big]=\psi_{1,p}(z)F[z,\xi_{j+1},\ldots,\xi_p],\quad j=0,1,\ldots\ .
\end{equation}
Then the  error in $R_{p,k}(z)$ has the determinant representation
\begin{equation}\label{eq:error}
F(z)-R_{p,k}(z)=\frac{\Delta(z)}{Q(z)},
\end{equation}
where
\begin{equation}\label{delta}
\Delta(z)=\begin{vmatrix}
 \Delta_0(z)& \Delta_1(z)&
\cdots&  \Delta_k(z)\\
u_{1,0}&u_{1,1}&\cdots&u_{1,k}\\
u_{2,0}&u_{2,1}&\cdots&u_{2,k}\\
\vdots&\vdots& &\vdots\\
u_{k,0}&u_{k,1}&\cdots&u_{k,k}\end{vmatrix},\quad
Q(z)=\begin{vmatrix}
 \psi_{1,0}(z)& \psi_{1,1}(z)&
\cdots&  \psi_{1,k}(z)\\
u_{1,0}&u_{1,1}&\cdots&u_{1,k}\\
u_{2,0}&u_{2,1}&\cdots&u_{2,k}\\
\vdots&\vdots& &\vdots\\
u_{k,0}&u_{k,1}&\cdots&u_{k,k}\end{vmatrix}.
\end{equation}

\end{lemma}

We next specialize Lemma \ref{le2} to suit the error formula for ITEA:

\begin{lemma}\label{le22} Let $ p>k+\deg u$. Define
\beq \label{Psi} \Psi_p(z)\equiv\psi_{1,p+k}(z).\eeq
Then the following are true whether the $\xi_i$ are distinct or not:
\begin{itemize}
\item [\textnormal{(i)}]
$D_{j+1,p+i}$ is given as in
\beq \label{DDD} D_{j+1,p+i}=-\sum^\mu_{s=1}v_s\psi_{p+i+1,p+k}(z_s)\frac{\psi_{1,j}(z_s)}{\Psi_p(z_s)}.
\eeq
Therefore, we also have
\beq\label{uuu}
u_{i,j}=(q,D_{j+1,p+i})=-\sum^\mu_{s=1}\alpha_{i,s}\frac{\psi_{1,j}(z_s)}{\Psi_p(z_s)};\quad
\alpha_{i,s}=(q,v_s)\psi_{p+i+1,p+k}(z_s).
\eeq
\item [\textnormal{(ii)}]
As for $\Delta_j(z)$ in \eqref{delta_j}, we have
\beq \label{Del} \Delta_j(z)=\psi_{1,p}(z)\sum^\mu_{s=1}\widehat{e}^{(p)}_s(z)\frac{\psi_{1,j}(z_s)}{\Psi_p(z_s)};\quad
\widehat{e}^{(p)}_s(z)=\frac{v_s}{z-z_s}\psi_{p+1,p+k}(z_s).
\eeq
\end{itemize}
\end{lemma}

Comparing $\Psi_p(z)$ in \eqref{Psi}, $u_{i,j}$ in \eqref{uuu}, and $\Delta_j(z)$ in \eqref{Del} with the analogous quantities for IMMPE in \cite{Sidi:2008:MTC}, we realize that they have the {\em same} algebraic structure.\footnote{Note that the error formula for $F(z)-R_{p,k}(z)$ in case of IMMPE is {\em precisely of the form} given in \eqref{delta_j}--\eqref{delta} of Lemma \ref{le:error}, but with {\em different}
$\Psi_p(z)$, $u_{i,j}$, and $\Delta_j(z)$; namely,  (i)\,$\Psi_p(z)=\psi_{1,p+1}(z)$, (ii)\,$u_{i,j}=\alpha_{i,s}\psi_{1,j}(z)/\Psi_p(z)$ with $\alpha_{i,s}=(q_i,v_s)$, and
 (iii)\,$\Delta_j(z)=\psi_{1,p}(z)\sum^\mu_{s=1}\widehat{e}^{(p)}_s(z){\psi_{1,j}(z_s)}/{\Psi_p(z_s)}$
 with
 $\widehat{e}^{(p)}_s(z)=v_s(z_s-\xi_{p+1})/(z-z_s)$. See \cite{Sidi:2008:MTC}.}
Therefore, we can now apply the techniques of \cite{Sidi:2008:MTC} verbatim, subject to suitable conditions having to do with ITEA.

\subsection{Algebraic structures of $Q(z)$, $\Delta(z)$, and $F(z)-R_{p,k}(z)$}
Below, we recall that $\Psi_p(z)$ is as in \eqref{Psi},  $u_{i,j}$ and $\alpha_{i,s}$ are as in \eqref{uuu}, and
$\Delta_j(z)$ and $\widehat{e}^{(p)}_{s}(z)$ are as in \eqref{Del}.
Applying  Theorems 3.6, 3.7, and 3.8 of \cite{Sidi:2008:MTC} verbatim to $Q(z)$, $\Delta(z)$, and $F(z)-R_{p,k}(z)$, respectively, we have the following:

\begin{theorem} [\cite{Sidi:2008:MTC}, Theorem 3.6]\label{th:Q}
Let $F(z)$ be the vector-valued rational function in \eqref{ratF},
and precisely as described in the first paragraph of this section,
with the notation therein. Define

\begin{equation}\label{eq:T}
T_{s_1,\ldots,s_k}=\begin{vmatrix}
\alpha_{1,s_1}&\alpha_{1,s_2}&\cdots&\alpha_{1,s_k}\\
\alpha_{2,s_1}&\alpha_{2,s_2}&\cdots&\alpha_{2,s_k}\\
\vdots&\vdots&&\vdots\\
\alpha_{k,s_1}&\alpha_{k,s_2}&\cdots&\alpha_{k,s_k}
\end{vmatrix}.
\end{equation}

Then, with $p>k+\deg(u)$,
\begin{equation}\label{eq:Q}
Q(z)=(-1)^k\sum_{1\leq s_1<s_2<\cdots<s_k\leq \mu} T_{s_1,\ldots,s_k}
V(z,z_{s_1},z_{s_2},\ldots,z_{s_k})\bigg[\prod^k_{i=1}\Psi_p(z_{s_i})\bigg]^{-1}.
\end{equation}
\end{theorem}

\begin{theorem}  [\cite{Sidi:2008:MTC}, Theorem 3.7] \label{th:delta}
Let $F(z)$ be the vector-valued rational function in \eqref{ratF},
and precisely as described in the first paragraph of this section,
with the notation therein. With $u_{i,j}$ and $\alpha_{i,s}$ as in \eqref{uuu},   and $\widehat{e}^{(p)}_{s}(z)$
as in \eqref{Del}, define
\begin{equation}\label{eq:That}
\widehat{T}^{(p)}_{s_0,s_1,\ldots,s_k}(z)=\begin{vmatrix}
\widehat{e}^{(p)}_{s_0}(z)&\widehat{e}^{(p)}_{s_1}(z)&\cdots&\widehat{e}^{(p)}_{s_k}(z)\\
\alpha_{1,s_0}&\alpha_{1,s_1}&\cdots&\alpha_{1,s_k}\\
\alpha_{2,s_0}&\alpha_{2,s_1}&\cdots&\alpha_{2,s_k}\\
\vdots&\vdots&&\vdots\\
\alpha_{k,s_0}&\alpha_{k,s_1}&\cdots&\alpha_{k,s_k}
\end{vmatrix}.
\end{equation}
Then,  with $p>k+\deg(u)$, we have
\begin{multline}\label{eq:delta}
\Delta(z)=(-1)^k\psi_{1,p}(z) \\ \times
\sum_{1\leq s_0<s_1<\cdots<s_k\leq \mu} \widehat{T}^{(p)}_{s_0,s_1,\ldots,s_k}(z)
V(z_{s_0},z_{s_1},\ldots,z_{s_k})\bigg[\prod^k_{i=0}\Psi_p(z_{s_i})\bigg]^{-1}.
\end{multline}
\end{theorem}

Finally, combining \eqref{eq:Q} and \eqref{eq:delta}
 in \eqref{eq:error}, we obtain a simple and elegant
expression for $F(z)-R_{p,k}(z)$. This is the subject of the following theorem.
\begin{theorem}  [\cite{Sidi:2008:MTC}, Theorem 3.8]\label{th:F-R}
For the error in $R_{p,k}(z)$,  with $p>k+\deg(u)$, we have the closed-form expression
\begin{multline}\label{eq:RRR}
F(z)-R_{p,k}(z)=\psi_{1,p}(z) \\
\times\frac
{\displaystyle\sum_{1\leq s_0<s_1<\cdots<s_k\leq \mu} \widehat{T}^{(p)}_{s_0,s_1,\ldots,s_k}(z)
V(z_{s_0},z_{s_1},\ldots,z_{s_k})\bigg[\prod^k_{i=0}\Psi_p(z_{s_i})\bigg]^{-1}}
{\displaystyle
\sum_{1\leq s_1<s_2<\cdots<s_k\leq \mu} T_{s_1,s_2,\ldots,s_k}
V(z,z_{s_1},z_{s_2},\ldots,z_{s_k})\bigg[\prod^k_{i=1}\Psi_p(z_{s_i})\bigg]^{-1}}.
\end{multline}
\end{theorem}

\noindent{\bf Remark:}\, When $k=\mu$ in Theorem \ref{th:F-R}, the summation
in the numerator on the right-hand side of \eqref{eq:RRR} is empty.  Thus, this theorem
provides an independent proof of the reproducing property of ITEA when $F(z)$ has only simple poles.

\section{Preliminaries for convergence theory} \label{se3a}
\setcounter{equation}{0}
\setcounter{theorem}{0}

Let $E$ be a closed and bounded set in the $z$-plane, whose complement
$K$, including the point at infinity, has a classical Green's function $g(z)$ with
a pole at infinity, which is continuous on $\partial E$, the boundary of $E$,
and is zero on $\partial E$.  For each $\sigma$, let $\Gamma_\sigma$ be the locus
$g(z)=\log\sigma$, and let $E_\sigma$ denote the interior of $\Gamma_\sigma$.
Then,  $E_1$ is the interior of $E$ and, for $1<\sigma<\sigma'$,
there holds $E\subset E_\sigma\subset E_{\sigma'}.$

For each  $p\in\{1,2,\ldots\},$ let
\begin{equation}\label{eq:xi}
 \Xi_p=\big\{\xi^{(p)}_1,\xi^{(p)}_2,\ldots,\xi^{(p)}_{p+k}\big\}
 \end{equation}
 be the set of interpolation points used in constructing the ITEA interpolant
 $R_{p,k}(z)$. Assume that the sets $\Xi_p$ are such that
 $\xi^{(p)}_i$ have no limit points in $K$ and
\begin{equation}\label{eq:Phi}
 \lim_{p\to\infty}\bigg|\prod^{p+k}_{i=1} \big(z-\xi^{(p)}_i\big)\bigg|^{1/p}=
\kappa\Phi(z);\quad
\kappa=\text{cap}\,(E),\quad \Phi(z)=\exp[g(z)],
\end{equation}
uniformly in $z$ on every compact subset of $K$,
where $\text{cap}(E)$ is the logarithmic capacity of $E$ defined by
$$ \text{cap}\,(E)=\lim_{n\to\infty}\big(\min_{r\in{\cal P}_n}\max_{z\in E}
|r(z)|\big)^{1/n};\quad {\cal P}_n=\big\{r(z):\ r\in\Pi_n \ \text{and monic}\big\}.$$
Such sequences $\big\{\xi^{(p)}_1,\xi^{(p)}_2,\ldots,\xi^{(p)}_{p+k}\big\},$ $p=1,2,\ldots,$ exist,
see Walsh \cite[p. 74]{Walsh:1960:IA}. Note that, in terms of $\Phi(z)$, the locus
$\Gamma_\sigma$ is defined by $\Phi(z)=\sigma$ for $\sigma>1$, while
$\partial E=\Gamma_1$ is simply the locus $\Phi(z)=1$.

 Recalling that $\prod^{p+k}_{i=1} \big(z-\xi^{(p)}_i\big)=\Psi_p(z)$ [see \eqref{Psi}],
we can write \eqref{eq:Phi} also as in
\begin{equation}\label{eq:phi-psi}
\lim_{p\to\infty}\big|\Psi_p(z)\big|^{1/p}=\kappa \Phi(z),
\end{equation}
uniformly in $z$ on every compact subset of $K$.\footnote{Note that the definition of $\Phi(z)$ for ITEA given in \eqref{eq:Phi} and \eqref{eq:phi-psi} is of {\em the same form} as the definition of $\Phi(z)$ for IMMPE, but the two differ;  for IMMPE, $\lim_{p\to\infty}\big|\prod^{p+1}_{i=1} \big(z-\xi^{(p)}_i\big)\big|^{1/p}=\lim_{p\to\infty}\big|\Psi_p(z)\big|^{1/p}=\kappa\Phi(z)$, where $\kappa=\text{cap}(E)$ as usual.}

It is clear that
\beq\label{qwe}z'\in \Gamma_{\sigma'},\quad z''\in \Gamma_{\sigma''}\quad\text{and}\quad
1<\sigma'<\sigma''\quad\Rightarrow\quad 1<\Phi(z')<\Phi(z'').\eeq

\section{Convergence theory for vector-valued rational $F(z)$ with simple poles}
\setcounter{equation}{0}
\setcounter{theorem}{0}
In this section, we provide a convergence theory, in case $F(z)$ is
a vector-valued rational function with simple poles as in \eqref{ratF}, for the sequences
$\{R_{p,k}(z)\}^\infty_{p=1}$ with  $k<\mu$ and fixed. [Note that by the reproducing
property mentioned in Section \ref{se1}, for $k=\mu$, $R_{p,k}(z)=F(z)$ for all $p\geq p_0$,
where $p_0-1$ is the degree of the numerator of $F(z)$.] Also, as we will let
$p\to\infty$ in our analysis, the condition that  $p>k+\deg(u)$, which  is necessary for
Theorem \ref{th:Q}, \ref{th:delta}, and \ref{th:F-R}, is satisfied
for all large $p$.

We continue to use the notation of the preceding sections.
We now turn to $F(z)$ in \eqref{ratF}. We assume that $F(z)$ is analytic in $E$.
This implies that its poles
$z_1,\ldots,z_\mu$ are all in $K$.
Now we order the poles of $F(z)$ such that
\begin{equation}\label{order}
\Phi(z_1)\leq \Phi(z_2)\leq \cdots\leq \Phi(z_\mu).
\end{equation}
By \eqref{qwe}, if $z'$ and $z''$ are two different poles of
$F(z)$, and $\Phi(z')<\Phi(z''),$
then $z'$ and $z''$ lie on two different loci $\Gamma_{\sigma'}$ and
$\Gamma_{\sigma''}$. In addition, $\sigma'<\sigma''$, that is, the set
$E_{\sigma'}$ is in the interior of $E_{\sigma''}$.

\subsection{Convergence analysis for $V_{p,k}(z)$}
We now state a K\"{o}nig-type convergence theorem for $V_{p,k}(z)(z)$  and another
theorem concerning its zeros. Since all our results eventually rely on the assumption that $T_{1,2,\ldots,k}\neq0$, we start by exploring the minimal conditions under which this assumption may hold for ITEA:

\begin{lemma}\label{le9} $T_{s_1,\ldots,s_k}$ is of the form
\beq \label{eqTTT}
T_{s_1,\ldots,s_k}=(-1)^{k(k-1)/2}V(z_{s_1},\ldots,z_{s_k})\prod^k_{i=1}(q,v_{s_i}).\eeq
\end{lemma}

\begin{proof} Invoking   $\alpha_{i,s}=(q,v_s)\psi_{p+i+1,p+k}(z_s)$ [see
\eqref{uuu}] in \eqref{eq:T}, and letting $\beta_i=(q,v_i)$ for simplicity of notation, we have
\beq\label{eqTTT1} T_{s_1,\ldots,s_k}=\begin{vmatrix}\beta_{s_1}\psi_{p+2,p+k}(z_{s_1})& \beta_{s_2}\psi_{p+2,p+k}(z_{s_2})&\cdots&\beta_{s_k}\psi_{p+2,p+k}(z_{s_k})\\
\beta_{s_1}\psi_{p+3,p+k}(z_{s_1})& \beta_{s_2}\psi_{p+3,p+k}(z_{s_2})&\cdots&\beta_{s_k}\psi_{p+3,p+k}(z_{s_k})\\
\vdots&\vdots&&\vdots\\
\beta_{s_1}\psi_{p+k+1,p+k}(z_{s_1})& \beta_{s_2}\psi_{p+k+1,p+k}(z_{s_2})&\cdots&\beta_{s_k}\psi_{p+k+1,p+k}(z_{s_k})\end{vmatrix},\eeq
which, upon factoring out $\beta_{s_1},\ldots,\beta_{s_k}$, becomes
\beq\label{eqTTT2} T_{s_1,\ldots,s_k}=T'_{s_1,\ldots,s_k}\prod^k_{i=1}\beta_{s_i},\eeq
 where
\beq\label{eqTTT3}T'_{s_1,\ldots,s_k}=\begin{vmatrix}\psi_{p+2,p+k}(z_{s_1})& \psi_{p+2,p+k}(z_{s_2})&\cdots&\psi_{p+2,p+k}(z_{s_k})\\
\psi_{p+3,p+k}(z_{s_1})& \psi_{p+3,p+k}(z_{s_2})&\cdots&\psi_{p+3,p+k}(z_{s_k})\\
\vdots&\vdots&&\vdots\\
\psi_{p+k+1,p+k}(z_{s_1})& \psi_{p+k+1,p+k}(z_{s_2})&\cdots&\psi_{p+k+1,p+k}(z_{s_k})\end{vmatrix}.\eeq
Now, $\psi_{p+i+1,p+k}(z)$ is a monic polynomial of degree $k-i$, $i=1,\ldots, k$. Therefore, after permuting the rows of the determinant $T'_{s_1,\ldots,s_k}$ suitably, we can apply Lemma \ref{le:a2}, and obtain
\beq\label{eqTTT4} T'_{s_1,\ldots,s_k}=(-1)^{k(k-1)/2}V(z_{s_1},\ldots,z_{s_k}).\eeq
This completes the proof.
\end{proof}

\noindent{\bf Remark:} Judging from \eqref{eqTTT1}--\eqref{eqTTT3}, we may be led to believe that
$T_{s_1,\ldots,s_k}$ is actually a function of $p$. The result in \eqref{eqTTT} shows that it is {\em independent of} $p$, and this is quite surprising.\\

Theorem \ref{th:den1} that follows concerns the convergence of $V_{p,k}(z)$ as $p\to\infty$.

\begin{theorem} [see \cite{Sidi:2008:MTC}, Theorem 5.1] \label{th:den1}
Assume
\begin{equation}\label{order-phi}
\Phi(z_k)<\Phi(z_{k+1})=\cdots=\Phi(z_{k+r})<\Phi(z_{k+r+1}),
\end{equation}
in addition to \eqref{order}.
In case $k+r=\mu,$ we define $\Phi(z_{k+r+1})=\infty$.
Assume also that
\begin{equation}\label{eq:T0}
\prod^k_{i=1}(q,v_i)\neq0.\eeq Consequently,
\begin{equation}\label{eq:T1}
T_{1,\ldots,k}\neq 0,
\end{equation}
and there holds
\begin{equation}\label{eq:Q1}
Q(z)=(-1)^kT_{1,\ldots,k}V(z,z_1,\ldots,z_k)\bigg[\prod^k_{i=1}\Psi_p(z_i)\bigg]^{-1}
\bigg[1+O\bigg(\frac{\Psi_p(z_k)}{\widetilde{\Psi}_{p,k}}\bigg)\bigg]\quad
\text{as $p\to\infty$},
\end{equation}
uniformly in every compact subset of $\C\setminus\{z_1,z_2,\ldots,z_k\},$
where
\begin{equation}\label{eq:tildepsi}
\big|\widetilde{\Psi}_{p,k}\big|=\min_{1\leq j\leq r}\big|\Psi_p(z_{k+j})\big|.
\end{equation}
Thus, with the normalization that $c_k=1$, and letting
\begin{equation}\label{eq:S}
S(z)=\prod^k_{i=1}(z-z_i),
\end{equation}
there holds
\begin{equation}\label{eq:Q2}
V_{p,k}(z)-S(z)=O\bigg(\frac{\Psi_p(z_k)}{\widetilde{\Psi}_{p,k}}\bigg)\quad
\text{as $p\to\infty$},
\end{equation}
from which we also have
\begin{equation}\label{eq:Q3}
\limsup_{p\to\infty}\big|V_{p,k}(z)-S(z)\big|^{1/p}\leq \frac{\Phi(z_k)}{\Phi(z_{k+1})}.
\end{equation}
\end{theorem}

Theorem \ref{th:den1} implies that $V_{p,k}(z)$ has precisely $k$ zeros that
tend to those of $S(z)$. Let us denote the zeros of $V_{p,k}(z)$ by $z^{(p)}_m$,
$m=1,\ldots,k.$ Then $\lim_{p\to\infty}z^{(p)}_m=z_m$,  $m=1,\ldots,k.$
In the next theorem, we provide the rate of convergence of each of these zeros.

\begin{theorem} [\cite{Sidi:2008:MTC}, Theorem 5.2]\label{th:den2}
Under the conditions of Theorem \ref{th:den1}, there holds
\begin{equation}\label{eq:z0}
z^{(p)}_m-z_m=O\bigg(\frac{\Psi_p(z_m)}
{\widetilde{\Psi}_{p,k}}\bigg)\quad\text{as $p\to\infty$,}
\end{equation}
with $\widetilde{\Psi}_{p,k}$ as in \eqref{eq:tildepsi}.
From this, it follows that
\begin{equation}\label{eq:z1}
\limsup_{p\to\infty}\big|z^{(p)}_m-z_m\big|^{1/p}\leq \frac{\Phi(z_m)}{\Phi(z_{k+1})},\quad
m=1,\ldots,k.
\end{equation}

In case $r=1$ in \eqref{order-phi}, that is,
\begin{equation}\label{eq:z20}
\Phi(z_k)<\Phi(z_{k+1})<\Phi(z_{k+2}),
\end{equation}
and assuming that $T_{1,\ldots,m-1,m+1,\ldots,k+1}\neq 0$, we have the more refined result
\begin{gather}
z^{(p)}_m-z_m\sim C_m\frac{\Psi_p(z_m)}{\Psi_p(z_{k+1})}\quad\text{as $p\to\infty$,}\nonumber\\
C_m=(-1)^{k-m}\frac{T_{1,\ldots,m-1,m+1,\ldots,k+1}}{T_{1,\ldots,k}}(z_{k+1}-z_m)
\prod^k_{\substack{i=1\\ i\neq m}}\frac{z_{k+1}-z_i}{z_m-z_i}.\label{eq:z10}
\end{gather}

\end{theorem}

\subsection{Convergence analysis for $R_{p,k}(z)$}
We now develop a de Montessus type convergence theory for the $R_{p,k}(z)$; that is, we analyze the error  $F(z)-R_{p,k}(z)$  as $p\to\infty$ with $k$  being held fixed.

We start by showing that
 the vectors $\widehat{T}^{(p)}_{s_0,s_1,\ldots,s_k}(z)$ are (i)\,meromorphic  in $z$ with simple poles at the $z_i$ and (ii)\,bounded for all large $p$.
This is the subject of the lemma that follows.

\begin{lemma}\label{le13}
For $z \not\in\{z_{s_0},z_{s_1}\ldots,z_{s_k}\}$, $\widehat{T}^{(p)}_{s_0,s_1,\ldots,s_k}(z)$ is analytic in $z$ and bounded for all large $p$.
\end{lemma}
\begin{proof}
Expanding the vector-valued determinant in  \eqref{eq:That} with respect to its first row, we obtain
\beq \label{eq:Th1} \widehat{T}^{(p)}_{s_0,s_1,\ldots,s_k}(z)=\sum^k_{j=0} E_j \widehat{e}^{(p)}_{s_j}(z),\eeq
where
\beq \label{eq:Th2} E_j=(-1)^jT_{s_0,\ldots,s_{j-1},s_{j+1},\ldots,s_k},\quad
\widehat{e}^{(p)}_{s_j}(z)=\frac{v_{s_j}}{z-z_{s_j}}\prod_{i=p+1}^{p+k}(z_{s_j}-\xi^{(p)}_i),\quad j=0,1,\ldots,k.\eeq
By Lemma \ref{le9}, $E_j$ are all scalars independent of $p$. In addition, $\widehat{e}^{(p)}_{s_j}(z)$ are bounded in $p$ since $\xi^{(p)}_{p+1},\ldots,\xi^{(p)}_{p+k}$ are bounded due to the assumption that the $\xi^{(p)}_i$ have no limit points in $K$, and $k$ is a fixed integer.
This completes the proof.
\end{proof}

We make use of Lemma \ref{le13} in the proof of Theorem \ref{th:num1}
 that follows.
Throughout the rest of this work, $\|Y\|$ denotes the vector norm
of $Y\in\C^N$.

\begin{theorem} [see \cite{Sidi:2008:MTC}, Theorem 5.3] \label{th:num1}
Under the conditions of Theorem \ref{th:den1}, $R_{p,k}(z)$ exists and is unique and satisfies
\begin{equation}\label{eq:R1}
F(z)-R_{p,k}(z)=O\bigg(\frac{\Psi_{p}(z)}{\widetilde{\Psi}_{p,k}}\bigg)
\quad \text{as $p\to\infty$,}
\end{equation}
uniformly on every compact subset of $\C\setminus\{z_1,\ldots,z_\mu\}$,
with $\widetilde{\Psi}_{p,k}$ as defined in \eqref{eq:tildepsi}. From  this, it also follows that
\begin{equation}\label{eq:R2}
\limsup_{p\to\infty}\big\|F(z)-R_{p,k}(z)\big\|^{1/p}\leq \frac{\Phi(z)}{\Phi(z_{k+1})},
\quad z\in \widetilde{K}=K\setminus\{z_1,\ldots,z_\mu\},
\end{equation}
uniformly on each compact subset of $\widetilde{K},$ and
\begin{equation}\label{eq:R3}
\limsup_{p\to\infty}\big\|F(z)-R_{p,k}(z)\big\|^{1/p}\leq \frac{1}{\Phi(z_{k+1})},
\quad z\in E,
\end{equation}
uniformly on $E$.
 Thus, uniform convergence takes place
for $z$ in any compact subset of the set $\widetilde{K}_k$, where
$$\widetilde{K}_k=\text{{\em int}}\,\Gamma_{\sigma_k}\setminus\{z_1,\ldots,z_k\};
\quad\sigma_k=\Phi(z_{k+1}).$$

When $r=1$ in \eqref{order-phi}, that is, when
\begin{equation}\label{eq:z200}
\Phi(z_k)<\Phi(z_{k+1})<\Phi(z_{k+2}),
\end{equation}
and  $\widehat{T}^{(p)}_{1,\ldots,k+1}(z)\neq 0$ in addition to \eqref{eq:T1}, we have the more refined result
\begin{gather}
F(z)-R_{p,k}(z)\sim B_p(z)\frac{\psi_{1,p}(z)}{\Psi_p(z_{k+1})}\quad\text{as $p\to\infty$,}\nonumber\\
B_p(z)=(-1)^{k}\frac{\widehat{T}^{(p)}_{1,\ldots,k+1}(z)}{T_{1,\ldots,k}}
\prod^k_{i=1}\frac{z_{k+1}-z_i}{z-z_i},\label{eq:z210}
\end{gather}
and  $B_p(z)$ is bounded for all large $p$.
\end{theorem}

\section{Convergence theory for general meromorphic $F(z)$ with simple poles}\label{se5}
\setcounter{equation}{0}
\setcounter{theorem}{0}
Let the  sets of interpolation points $\{\xi^{(p)}_1,\ldots,
\xi^{(p)}_{p+k}\}$ be as in the
preceding section.
We now turn to the convergence analysis of  $R_{p,k}(z)$ as
$p\to\infty$, when the function $F(z)$ is analytic in $E$ and
meromorphic in  $E_\rho=\text{int}\,\Gamma_\rho$, where
$\Gamma_\rho,$ as before,  is the locus $\Phi(z)=\rho$ for some $\rho>1$.
Assume that $F(z)$ has $\mu$ simple poles $z_1,\ldots,z_\mu$ in $E_\rho$.
Thus, $F(z)$ has the following form:
\begin{equation}\label{merF}
F(z)=\sum^\mu_{s=1}\frac{v_s}{z-z_s}+\Theta(z),
\end{equation}
$\Theta(z)$ being analytic in $E_\rho$.

The treatment of this case is based entirely on that of the preceding section,
the differences being minor. Note that the polynomial $u(z)$ of \eqref{ratF}
is now replaced by $\Theta(z)$ in \eqref{merF}. Previously, we had
$u[\xi_m,\ldots,\xi_n]=0$ for all large $n-m$, as a consequence  of which, we had
  \eqref{uuu} for $u_{i,j}$ and  \eqref{Del} for $\Delta_j(z)$.
Instead of these, we now have
\begin{equation}\label{eq:uij1}
u_{i,j}=
-\sum^\mu_{s=1}\alpha_{i,s}\frac{\psi_{1,j}(z_s)}{\Psi_p(z_s)}+
\big(q,\Theta[\xi_{j+1},\ldots,\xi_{p+i}]\big),
\end{equation}
with $\alpha_{i,s}$ as in \eqref{uuu}, and
\begin{equation}\label{eq:deltaj1}
\Delta_j(z)=\psi_{1,p}(z)\bigg(\sum^\mu_{s=1}\widehat{e}^{(p)}_s(z)\frac{\psi_{1,j}(z_s)}{\Psi_{p}(z_s)}
+\Theta[z,\xi_{j+1},\ldots,\xi_{p}]\bigg),
\end{equation}
with $\widehat{e}^{(p)}_s(z)$ as in \eqref{Del}.

It is clear that the treatment of the general meromorphic $F(z)$
will be the same as that of the
rational $F(z)$ provided the contributions from $\Theta(z)$ to $u_{i,j}$ and
$\Delta_j(z)$, as $p\to\infty$,  are negligible compared to the rest of the terms
in \eqref{eq:uij1} and \eqref{eq:deltaj1}. This is indeed the case, as is shown in
\cite[Lemma 6.1]{Sidi:2008:MTC}:
\begin{lemma}[\cite{Sidi:2008:MTC}, Lemma 6.1]\label{le:theta11}
With $F(z)$ as in the first paragraph, there holds
\begin{equation}\label{eq:theta111}
\limsup_{p\to\infty}\big\|\Theta[\xi^{(p)}_{j+1},\ldots,\xi^{(p)}_{p+i}]\big\|^{1/p}\leq \frac{1}{\kappa\rho}.
\end{equation}
There also holds
\begin{equation}\label{eq:theta112}
\limsup_{p\to\infty}\big\|\Theta[ z,\xi^{(p)}_{j+1},\ldots,\xi^{(p)}_{p}]\big\|^{1/p}\leq \frac{1}{\kappa\rho},
\end{equation}
uniformly in every compact subset of $E_\rho.$ These hold for all  $i\leq k$ and $j\leq k$.
\end{lemma}

 With this information, we can now prove
convergence results for $V_{n,k}(z)$ and $F(z)-R_{p,k}(z)$ for general meromorphic $F(z)$.
We recall that the poles $z_1,\ldots,z_\mu$ of $F(z)$ are ordered such that
\begin{equation}\label{order-p}
 \Phi(z_1)\leq \Phi(z_2)\leq \cdots\leq \Phi(z_\mu)\leq\rho.
 \end{equation}
We also adopt the notation of Theorems
\ref{th:den1}, \ref{th:den2}, and \ref{th:num1}.

\begin{theorem}[see \cite{Sidi:2008:MTC}, Theorem 6.2]\label{th:den11}
{\em (i)}\,When $k<\mu$, assume that
\begin{equation}\label{order-phi-p}
\Phi(z_k)<\Phi(z_{k+1})=\cdots=\Phi(z_{k+r})<\begin{cases}
\Phi(z_{k+r+1})&\quad\text{if $k+r<\mu$},\\ \rho &\quad\text{if $k+r=\mu$,}\end{cases}
\end{equation}
in addition to \eqref{order-p}.
Assume also that
\begin{equation}\label{eq:T00}
\prod^k_{i=1}(q,v_i)\neq0.\eeq Consequently,
\begin{equation}\label{eq:T11}
T_{1,\ldots,k}\neq 0,
\end{equation}
Then, all the results of Theorem \ref{th:den1}  hold.

{\em (ii)}\, When $k=\mu$,
\begin{equation}\label{eq:Q3-p}
\limsup_{p\to\infty}\big|V_{p,k}(z)-S(z)\big|^{1/p}\leq \frac{\Phi(z_k)}{\rho}.
\end{equation}
uniformly on every compact subset of $\C\setminus\{z_1,\ldots,z_\mu\}$.
\end{theorem}

Theorem \ref{th:den11} implies that $V_{p,k}(z)$ has precisely $k$ zeros that
tend to those of $S(z)$. Let us denote the zeros of $V_{p,k}(z)$ by $z^{(p)}_m$,
$m=1,\ldots,k.$ Then $\lim_{p\to\infty}z^{(p)}_m=z_m$,  $m=1,\ldots,k.$
In the next theorem, we provide the rate of convergence of each of these zeros.

\begin{theorem} [\cite{Sidi:2008:MTC}, Theorem 6.3]\label{th:den22}
Assume the  conditions of Theorem \ref{th:den2}.

{\em (i)}\,When $k<\mu$,  all the results of Theorem \ref{th:den2} hold.

{\em (ii)}\,When $k=\mu$,
\begin{equation}\label{eq:z1-p}
\limsup_{p\to\infty}\big|z^{(p)}_m-z_m\big|^{1/p}\leq \frac{\Phi(z_m)}{\rho},\quad
m=1,\ldots,k.
\end{equation}
\end{theorem}

Our next and last result concerns the convergence of $R_{p,k}(z)$:
\begin{theorem}[\cite{Sidi:2008:MTC}, Theorem 6.4]\label{th:num11}
Assume the conditions of Theorem \ref{th:num1}. Then $R_{p,k}(z)$ exists and is unique.

{\em (i)}\,When $k<\mu$,  all the results of Theorem \ref{th:num1} hold
with $\widetilde{K}=E_\rho\setminus\{z_1,\ldots,z_\mu\}$.

{\em (ii)}\,When $k=\mu$, there holds

\begin{equation}\label{eq:R2-p}
\limsup_{p\to\infty}\big\|F(z)-R_{p,k}(z)\big\|^{1/p}\leq \frac{\Phi(z)}{\rho},
\quad z\in \widetilde{K}=E_\rho\setminus\{z_1,\ldots,z_\mu\},
\end{equation}
uniformly on each compact subset of $\widetilde{K},$ and
\begin{equation}\label{eq:R3-p}
\limsup_{p\to\infty}\big\|F(z)-R_{p,k}(z)\big\|^{1/p}\leq \frac{1}{\rho},
\quad z\in E,
\end{equation}
uniformly on $E$.
\end{theorem}


\begin{thebibliography}{10}

\bibitem{Atkinson:1989:INA}
K.E. Atkinson.
\newblock {\em {An Introduction to Numerical Analysis}}.
\newblock John Wiley \& Sons Inc., New York, second edition, 1989.

\bibitem{Brezinski:1975:GTS}
C.~Brezinski.
\newblock G\'{e}n\'{e}ralisations de la transformation de {Shanks}, de la table
  de {Pad\'{e}}, et de l'$\epsilon$-algorithme.
\newblock {\em Calcolo}, 12:317--360, 1975.

\bibitem{Sidi:1994:RAP}
A.~Sidi.
\newblock Rational approximations from power series of vector-valued
  meromorphic functions.
\newblock {\em J. Approx. Theory}, 77:89--111, 1994.

\bibitem{Sidi:2004:NAV}
A.~Sidi.
\newblock A new approach to vector-valued rational interpolation.
\newblock {\em J. Approx. Theory}, 130:177--187, 2004.

\bibitem{Sidi:2006:APS}
A.~Sidi.
\newblock Algebraic properties of some new vector-valued rational interpolants.
\newblock {\em J. Approx. Theory}, 141:142--161, 2006.

\bibitem{Sidi:2008:MTC}
A.~Sidi.
\newblock A de {Montessus} type convergence study for a vector-valued rational
  interpolation procedure.
\newblock {\em Israel J. Math.}, 163:189--215, 2008.

\bibitem{Sidi:2008:IMPE-1}
A.~Sidi.
\newblock A de {Montessus} type convergence study of a least-squares
  vector-valued rational interpolation procedure.
\newblock {\em J. Approx. Theory}, 155:75--96, 2008.

\bibitem{Sidi:2010:IMPE-2}
A.~Sidi.
\newblock A de {Montessus} type convergence study of a least-squares
  vector-valued rational interpolation procedure {II}.
\newblock {\em Comput. Methods Funct. Theory}, 10:223--247, 2010.

\bibitem{Stoer:2002:INA}
J.~Stoer and R.~Bulirsch.
\newblock {\em {Introduction to Numerical Analysis}}.
\newblock Springer-Verlag, New York, third edition, 2002.

\bibitem{Walsh:1960:IA}
J.L. Walsh.
\newblock {\em Interpolation and {Approximation}}, volume~20 of {\em American
  Mathematical Society Colloquium Publications}.
\newblock American Mathematical Society, Providence, Rhode Island, third
  edition, 1960.

\end{thebibliography}

\end{document}